\documentclass{jgcc}
\pdfoutput=1
\usepackage{lastpage}
\jgccdoi{18}{1}{1}{17174}
\jgccheading{}{\pageref{LastPage}}{}{}{Dec.~23,~2025}{Mar.~17,~2026}{}  


\usepackage{amsfonts}
\usepackage{amsmath,amssymb}

\usepackage{hyperref}

\newcommand{\astA}{\underset{A}{\ast}}
\newcommand{\rk}{\mathop{\rm rk}\nolimits}
\let\phi=\varphi
\let\precc=\preccurlyeq

\theoremstyle{plain} 

\def\cf{{\em cf.}}
\def\ie{{\em i.e.}}

\begin{document}

\title[Subgroups of Cyclically Amalgamated Free Products]{Subgroups of Cyclically Amalgamated Free Products}

\author[M.~Kreuzer]{Martin Kreuzer}	
\address{Faculty of Computer Science and Mathematics, University of Passau, 94030 Passau, Germany}
\email{Martin.Kreuzer@uni-passau.de}

\author[A.~Moldenhauer]{Anja Moldenhauer}	
\address{Hammer Steindamm 111, 20535 Hamburg, Germany}	
\email{A.Moldenhauer@gmx.net}

\author[G.~Rosenberger]{Gerhard Rosenberger}	
\address{Department of Mathematics, University of Hamburg, Bundesstr. 55, 20146 Hamburg, Germany}
\email{Gerhard.Rosenberger@uni-hamburg.de}




\begin{abstract}
\noindent 
Given a group $G = H_1 \astA H_2$ which is the free product of 
two finitely generated groups $H_1$ and $H_2$ with amalgamation
over a cyclic subgroup~$A$ which is malnormal in~$G$, we study relations between
the structure of its subgroups and the structure of the group~$G$ itself.
Firstly, we show that if $H_1$ and $H_2$ are 3-free products of cyclics of rank $\ge 3$ 
then $G$ is also a 3-free product of cyclics. Secondly, we prove that if $H_1$ and $H_2$ 
are 4-free products of cyclics of rank $\ge 4$
then every 4-generated subgroup of $G$ is a free product of $\le 4$ cyclics or a 1-relator
quotient of a free product of four cyclic groups.
Here a group is called an $n$-free product of cyclics if every $n$-generated subgroup
is a free product of $\le n$ cyclic groups. These results are based
on ubiquitous applications of the Nielsen method for amalgamated free products
which we recall carefully.

Lastly, given an infinite, finitely presented group which is not free, but all of its 
infinite index subgroups are free, a well-known conjecture says that it is isomorphic to a surface group. 
We revisit and elaborate on predominantly group theoretic proofs of this conjecture 
for cyclically amalgamated products as above, as well as for certain HNN extensions.
\end{abstract}

\maketitle
\hfill{\small \it Dedicated to Alexei Miasnikov}

\hfill{\small \it on the occasion of his birthday.}

\section{Introduction}\label{sec1}

Amalgamated free products arise naturally in group theory and other branches of mathematics.
For instance, surface groups (\ie, fundamental groups of compact 2-dimensional manifolds),
various knot groups, and the well-known Baumslag-Solitar groups are of this type. 
For two factors, amalgamated free products were first introduced in 1927 by O.~Schreier~\cite{Sch}.
For several factors, this was later generalized by H.\ Neumann~\cite{Neu1,Neu2}.
W.\ Magnus was one of the first mathematicians who recognized their value and used them in the
early 1930s in several papers on one-relator groups~\cite{Mag1,Mag2}.

It is an important task to describe the subgroups of amalgamated free products and the
relation between their structure and properties of the group itself. A first major step in this direction 
was taken by A.\ Karrass and D.\ Solitar in~\cite{KS}, where the authors give a very general 
and rather abstract classification which is, however, not easy to apply in concrete situations.
In this paper we examine the structure of subgroups with a small number of generators 
and the existence of free subgroups and free products of cyclic groups in the following type of group.
Let $G=H_1 \astA H_2$, where $H_1,H_2$ are finitely generated groups, where $A$ is an infinite cyclic group
which injects properly into~$H_1$ and~$H_2$ and is a malnormal subgroup of~$G$,
and where we assume that there exist $u\in H_1$, $v\in H_2$ with $A=\langle u\rangle = \langle v\rangle$
such that $uv$ involves all generators of~$G$. For simplicity, we shall say that~$G$ is a
\textit{cyclically amalgamated free product}.

In~\cite{FRW}, the authors showed that, for $n\in \{2,3\}$, every $n$-generated
subgroup in certain groups of F-type is a free product of cyclics. Later it was
discovered that this classification can be suitably extended to 4-generated subgroups.
In Section~\ref{sec3} we generalize these results as follows (see Proposition~\ref{prop:2-gen}
and Theorems~\ref{thm:3-gen} and~\ref{thm:4-gen}).

\smallskip
\begin{thm}
Let $G=H_1 \astA H_2$ be a cyclically amalgamated free product.
\begin{enumerate}
\item[(a)] If $H_i$ satisfies $\rk(H_i)\ge 2$ and is a 2-free product of cyclics for $i=1,2$,
then every 2-generated subgroup of~$G$ is a free product of $\le 2$ cyclics.

\item[(b)] If $H_i$ satisfies $\rk(H_i)\ge 3$ and is a 3-free product of cyclics for $i=1,2$,
then every 3-generated subgroup of~$G$ is a free product of $\le 3$ cyclics.

\item[(c)] If $H_i$ satisfies $\rk(H_i)\ge 4$ and is a 4-free product of cyclics for $i=1,2$,
then every 4-generated subgroup of~$G$ is a free product of $\le 4$ cyclics or a one-relator quotient of
a free product of four cyclics.
\end{enumerate}
\end{thm}

Here we say that a group is an \textit{$n$-free product of cyclics}
if every $n$-generated subgroup of~$G$ is a free product of at most~$n$ cyclic groups.
Notice that this notion is entirely different from the the classical notion
of \textit{$n$-free groups}, as introduced by G.\ Higman in 1951 (\cf~\cite{Hig}).
In particular, the factors of a product of at most $n$ cyclic groups 
are allowed to be finite.

Moreover, we point out that case (a) of this theorem has been known for a long time
(see~\cite[Thm.~6]{KS2}) and is included here merely for completeness sake.

In Section~\ref{sec2} we provide the main tool to prove this theorem, 
namely the method of Nielsen reductions in amalgamated free products. 
In the context of free groups, Nielsen reductions were first introduced (1921) in~\cite{Nie}. 
Later the method was adapted to amalgamated free products
by H.\ Zieschang in~\cite{Zie1}. It was further refined
by the third author in~\cite{Ros1,Ros2} and, together with R.N.\ Kalia, in~\cite{KR}. 
Based on the normal form of elements of amalgamated free products introduced 
already (1927) in~\cite{Sch}, we briefly recall the \textit{symmetric normal form} (see Proposition~\ref{prop:symmetric})
and the resulting preorder $\precc$ which refines the length (see Definition~\ref{def:preorder}). 
This allows us to introduce Nielsen reduced sets and to formulate the Nielsen method for
cyclically amalgamated free products in Theorem~\ref{thm:3cases} and Proposition~\ref{prop:refine}.
Explicit algorithms and a careful complexity analysis for the computation of normal forms
in amalgamated free products were given in 2006 by A.V.\ Borovik, A.\ Miasnikov, and V.N.\ 
Remeslennikov (see~\cite{BMR}).

Next we discuss low rank subgroups of cyclically amalgamated free products in Section~\ref{sec3}.
Besides proving the above theorem, we also apply it to recover the known results about
low rank subgroups of groups of F-type (see Corollary~\ref{cor:F-type}).
In the proof of these results, the method of Nielsen reductions plays a central role,
as it allows us to keep a tight control over the structure of the words representing the group
elements under consideration. It has been suggested that, possibly, the method of groups acting on
trees, as developed by H.\ Bass and J.P.\ Serre in~\cite{BS}, and in particular in the form 
elaborated by R.\ Weidmann (\cf~\cite{Wei}), could be used to obtain a proof of the above theorem.
Except for case~(a) of the theorem, which has been previously known, this is highly doubtful, 
as in this case the Bass-Serre method
encounters certain obstructions which are hard to analyze in detail (see~\cite[Thm.~2]{Wei})
and do not lead to explicit presentations of the subgroups under consideration
in a conceivable way.

In Section~\ref{sec4} we turn to free subgroups of infinite index. 
The discussion here should be considered as a brief survey of some direct
methods to analyze these subgroups for amalgamated free products and HNN extensions 
using group theoretic (rather than geometric or topological) methods.

The central conjecture here is the \textit{Surface Group
Conjecture} which asks whether an infinite, finitely presented, non-free group, such that every subgroup
of infinite index is free, has to be isomorphic to a surface group.
A major advance in this area was recently achieved by H.\ Wilton who proved the conjecture
(among others) for one-relator groups using heavy topological machinery~\cite{Wil2}.
In Theorem~\ref{thm:SGCforAFP}, we point out a proof in the case of cyclically amalgamated
free products which relies chiefly on group theoretic methods.

Closely related to amalgamated free products is the case of HNN extensions.
In Theorem~\ref{thm:23HNN}, we recall what is known about low rank subgroups
of HNN extensions of the type $G = \langle t,H \mid tut^{-1}=v\rangle$,
and in Theorem~\ref{thm:SGCforHNN} we prove the surface group conjecture
for such groups using chiefly group theoretic means.

Throughout we have strived to keep this paper as self-contained as possible.
All definitions and statements used without reference can be found
in~\cite{FMRSW}.

%
%

\section{Nielsen Reduction in Amalgamated Free Products}\label{sec2}

In this section we describe the Nielsen reduction method in amalgamated free products.
This method was originally developed by H.\ Zieschang in~\cite{Zie1} and then refined
by the third author in~\cite{Ros1,Ros2} and, together with R.N.\ Kalia, in~\cite{KR}. 

Given two groups $H_1,H_2$ and a group~$A$ which injects into both of them
via maps $\phi:\; A\rightarrow H_1$ and $\psi:\; A \rightarrow H_2$,
their \textbf{amalgamated free product} (or \textbf{free product with
amalgamation}) is $G= H_1 \astA H_2 = (H_1 \ast H_2)/N$ where~$N$ is the normal 
closure of the set of elements $\phi(a)\psi^{-1}(a)$ with $a\in A$.
We identify $H_1$, $H_2$, and $A=H_1\cap H_2$ with their images in~$G$ and
assume that $A\ne H_1$ as well as $A\ne H_2$. Notice that $G = H_1 \ast H_2$
is the free product of~$H_1$ and~$H_2$ if $A=\{1\}$ is the trivial group.

In this paper we consider the following type of amalgamated free product.

\begin{asm}\label{ass:A}
Let $H_1,H_2$ be finitely generated groups, and let $A$ be an infinite
cyclic group which injects properly into~$H_1$ and~$H_2$ via $\phi:\; A \rightarrow H_1$
and $\psi:\; A \rightarrow H_2$. Let $G= H_1 \astA H_2$ be the amalgamated free product
of~$H_1$ and~$H_2$. Assume that the following conditions are satisfied.
\begin{enumerate}
\item[(1)] Given a generator~$e$ of~$A$, let $u=\phi(e)$ and $v=\psi(e)$. Then~$uv$ 
involves all generators of~$G$.

\item[(2)] The image of~$A$ in~$G$ is a malnormal subgroup.
\end{enumerate}
\end{asm}

For simplicity, we shall say that $G$ is a \textbf{cyclically amalgamated free product}
if the assumption is satisfied.

Next we fix a system~$L_i$ of left coset representatives of~$A$ in~$H_i$
for $i=1,2$. Here we represent~$A$ by~$1$. The following proposition is
well-known (see~\cite[Thm. 1.5.8 and Cor. 1.5.10]{FMRSW}).

\begin{prop}[The Reduced Form]$\mathstrut$\label{prop:reduced}\\
Every element $g\in G$ has a unique representation $g = h_1 h_2 \cdots h_m a$,
where $a\in A$ and the elements $h_i$ are alternatingly from $L_1\setminus \{1\}$
and $L_2\setminus \{1\}$. This representation is called the \textbf{reduced form}
of~$G$.
\end{prop}

For $g\in G$ with reduced form $g= h_1 \cdots h_m a$, the number~$m$ is called 
the \textbf{length} of~$g$. For our purposes, the partial ordering for~$G$ 
given by the length is too coarse. Thus we have to refine it. 
For $i=1,2$, we take the elements of~$L_i^{-1}$
as a system of representatives for the right cosets of~$A$ in~$H_i$.
Now we have the following normal form (see~\cite[Sect.~1.5.2]{FMRSW} or~\cite[Prop.~2.3.12]{ZVC}).

\begin{prop}[The Symmetric Normal Form]$\mathstrut$\label{prop:symmetric}\\
Every element $g\in G$ has a unique representation 
\[
g \;=\; \ell_1 \cdots \ell_m \;k \; r_m \cdots r_1
\]
with $k\in H_1\cup H_2$, with $\ell_i$ alternatingly from $L_1\setminus \{1\}$ and
$L_2 \setminus \{1\}$, and with $r_i$ alternatingly from $L_1^{-1} \setminus \{1\}$
and $L_2^{-1}\setminus \{1\}$.

Moreover, if $k\in A$ then $\ell_m$ and $r_m$ belong to different
subgroups~$H_i$ (provided $m\ge 1$), and if $k\in H_i\setminus A$ then $\ell_m,r_m \notin H_i$ 
(provided $m\ge 1$).

This representation is called the \textbf{symmetric normal form} of~$g$.
The product $\ell_1\cdots \ell_m$ is called its \textbf{leading half}, the product
$r_m \cdots r_1$ is called its \textbf{rear half}, and the element~$k$ is known as the
\textbf{kernel} of~$g$.
\end{prop}

For the definition of the \textbf{length} $\lambda(g)$ of a symmetric normal form, we use 
$\lambda(g)=2m$ if $k\in A$ and $\lambda(g)=2m+1$ if $k\notin A$.
One advantage of the symmetric normal form is that in forming products, cancellations
are usually reduced to free cancellations.

Our next goal is to introduce a preorder relation on~$G$. By Assumption~\ref{ass:A},
the group~$G$ is countable and we can fix an enumeration.
For $i\in\{1,2\}$, we obtain a strict total ordering $<_i$ on the set~$L_i$
defined by this enumeration. In particular, for every $h\in L_i$ there exist
only finitely many $h'\in L_i$ such that $h'<_i h$.

\begin{rem}[Ordering the Leading Halves]$\mathstrut$\label{rem:orderL}
\begin{enumerate}
\item[(a)] Let $\le_L$ be the ordering relation on $L_1 \cup L_2$ defined by
\begin{enumerate}
\item[(1)] $\ell \le_L \ell$ for $\ell\in L_1\cup L_2$,
\item[(2)] $\ell_1 \le_L \ell_2$ for $\ell_1\in L_1$ and $\ell_2\in L_2$, and
\item[(3)] $\ell \le_L \ell'$ if $\ell,\ell'\in L_i$ and $\ell <_i \ell'$.
\end{enumerate}

\item[(b)] Now we extend the ordering relation $\le_L$ to the leading halves 
$\ell_1\cdots \ell_m$ of symmetric normal forms as follows:
\begin{enumerate}
\item[(1)] $\ell_1\cdots \ell_m \le_L \ell_1 \cdots \ell_m$,
\item[(2)] $\ell_1 \cdots \ell_m <_L \ell'_1 \cdots \ell'_{m'}$ if $m<m'$,
\item[(3)] $\ell_1 \cdots \ell_m <_L \ell'_1 \cdots \ell'_m$
if $\ell_1=\ell'_1$, $\dots$, $\ell_{i-1}= \ell'_{i-1}$, $\ell_i <_L \ell'_i$
for some $i\in \{1,\dots,m\}$. 
\end{enumerate}
In other words, we order the leading halves by the \textit{length-lex\-i\-co\-graphic
ordering} induced by the ordering $\le_L$ on $L_1\cup L_2$.

\item[(c)] The ordering $\le_L$ on the leading halves has the following properties:
\begin{enumerate}
\item[(1)] If $\ell_1\cdots \ell_{k-1} <_L \ell'_1 \cdots \ell'_{k-1}$ and
$\ell_k,\ell'_k$ are in the correct sets $L_i\setminus \{1\}$, then
$\ell_1 \cdots \ell_k <_L \ell'_1 \cdots \ell'_k$. 

\item[(2)] For a leading half $\ell_1 \cdots \ell_k$, there exist only finitely
many leading halves $\ell_1\cdots \ell_{k-1} \ell_k'$ with
$\ell_1\cdots \ell_{k-1}\ell'_k <_L \ell_1 \cdots \ell_k$.
\end{enumerate}

\item[(d)] To define an ordering relation $\le_R$ on the rear halves, we let
$r_m \cdots r_1 \le_R r'_{m'} \cdots r'_1$ $\Longleftrightarrow$
$(r'_1)^{-1} \cdots (r'_{m'})^{-1} \le_L r_1^{-1} \cdots r_m^{-1}$. 

\end{enumerate}
\end{rem}

At this point we are ready to define the desired preorder relation on~$G$.

\begin{defi}\label{def:preorder}
Let $G = H_1 \astA H_2$ be an amalgamated product as above.
We define a preorder $\precc$ on sets $\{g,g^{-1}\}$ with $g\in G$ as follows.
\begin{enumerate}
\item[(a)] Let $g'$ be the element in $\{g,g^{-1}\}$ such that
the leading half of~$g'$ is less than or equal to the leading half of~$(g')^{-1}$
with respect to~$\le_L$.

\item[(b)] Now we let $\{g,g^{-1}\} \precc \{h,h^{-1}\}$ if
\begin{enumerate}
\item[(1)] $\lambda(g') < \lambda(h')$,

\item[(2)] $\lambda(g')=\lambda(h')$, and the leading halves $\ell_1\cdots \ell_m$
of~$g'$ as well as $\ell'_1 \cdots \ell'_{m'}$ of~$h'$ satisfy $\ell_1 \cdots \ell_m
<_L \ell'_1 \cdots \ell'_{m'}$, or

\item[(3)] $\lambda(g') = \lambda(h')$, the leading halves of~$g'$ and~$h'$
are equal, and the rear halves $r_m \cdots r_1$ of~$g'$ and $r'_{m'}\cdots r'_1$ of~$h'$
satisfy $r_m \cdots r_1 \le_L r'_{m'}\cdots r'_1$.
\end{enumerate}

To simplify the notation, we shall also write $g' \precc h'$
instead of $\{g,g^{-1}\} \precc \{h,h^{-1}\}$ if no confusion can occur.

\end{enumerate}
\end{defi}

Notice that $g' \precc h'$ and $h' \precc g'$
imply that~$g$ and~$h$ differ only in their kernel. Since this can
occur with $g\ne h$, the relation $\precc$ is only a preorder, \ie, reflexive
and transitive, but in general not antisymmetric.
Next, recall that a \textbf{regular Nielsen transformation} of a set $\{g_1,\dots,g_n\}$
of elements of a group is a finite product of \textbf{elementary Nielsen transformations},
\ie, of transformations of one of the following types:
\begin{enumerate}
\item[(1)] Replace one element $g_i$ by $g_i^{-1}$.

\item[(2)] Replace one element $g_i$ by $g_i g_j$, where $j\ne i$.
\end{enumerate}
Moreover, two sets are called \textbf{Nielsen equivalent} if one can be transformed into the
other by a regular Nielsen transformation. The goal of Nielsen reduction is to
shorten sets of elements in~$G$ in the following sense.

\begin{defi}
Let $\{g_1,\dots,g_n\}$ and $\{h_1,\dots,h_n\}$ be two sets of elements in~$G$.
\begin{enumerate}
\item[(a)] The set $\{g_1,\dots,g_n\}$ is called \textbf{shorter} than the set $\{h_1,\dots,h_n\}$
if $g'_i \precc h'_i$ for $i=1,\dots,n$ and if $h'_j \precc g'_j$
fails for at least one $j\in \{1,\dots,n\}$.

\item[(b)] The set $\{g_1,\dots,g_n\}$ is called \textbf{Nielsen reduced} or
\textbf{$\precc$-minimal} if it is not Nielsen equivalent to a set
$\{h_1,\dots,h_n\}$ which is shorter or satisfies $h_i=1$ for some $i\in \{1,\dots,n\}$.

\end{enumerate}
\end{defi}

Thus our goal is to transform a given finite set of elements of~$G$ to a Nielsen reduced set
and to study the resulting $\precc$-minimal sets. 
The following rules are based on the
length and the ordering of elements in~$G$. They were first shown in~\cite[\S 1]{Zie1}
and later refined in~\cite{Ros1}. Easily accessible proofs 
are contained in~\cite[Lemmas 1.5.21 and 1.5.22]{FMRSW}.

\begin{prop}
Let $x,y,z\in G \setminus \{1\}$ be elements such that the following conditions are satisfied:
\begin{enumerate}
\item[(1)] $\lambda(xy) \ge \max\{ \lambda(x), \lambda(y) \}$ and 
$\lambda(yz) \ge \max\{ \lambda(y), \lambda(z)\}$,

\item[(2)] $xy\not\precc x$ and $yz\not\precc z$.
\end{enumerate}
Then the following claims hold.
\begin{enumerate}
\item[(a)] If $\lambda(xyz) \le \lambda(x) - \lambda(y) + \lambda(z)$ then~$y$ is conjugate to
an element of $H_1\cup H_2$.

\item[(b)] If $\lambda(y) < \min\{ \lambda(x), \lambda(z) \}$ and 
$\lambda(xyz) < \lambda(x) - \lambda(y) + \lambda(z)$ then~$y$ is conjugate to an element of~$A$.

\item[(c)] Let $\{g_1,\dots,g_n\}$ be Nielsen reduced, and let $x,y,z \in \{g_1,\dots,g_n,
g_1^{-1},\allowbreak \dots,g_n^{-1}\}$. If $\lambda(xyz) < \lambda(x) - \lambda(y) + \lambda(z)$ then~$y$
is conjugate to an element of~$A$, or we have $x=y=z$.
\end{enumerate}
\end{prop}

In view of the second case in~(c) we note that $\lambda(x^n) < \lambda(x)$ for some $n\ge 2$
implies that~$x$ is conjugate to an element in $H_1\cup H_2$ and $x^n$ is conjugate 
to an element of~$A$.

\begin{prop}
Let $x,z\in G$ be elements such that $\lambda(x)\ge 1$ and $\lambda(z)\ge 1$, and let 
$y\in A \setminus \{1\}$ be such that the following conditions are satisfied:
\begin{enumerate}
\item[(1)] $\lambda(xy) \ge \lambda(x)$ and $\lambda(yz) \ge \lambda(z)$,

\item[(2)] $xy \not \precc x$ and $yz \not\precc z$,

\item[(3)] $\lambda(xyz) < \lambda(x) - \lambda(y) + \lambda(z) - 1 = \lambda(x) +\lambda(z)-1$.
\end{enumerate}
Then there exists an element $h\in (H_1\cup H_2) \setminus A$ such that $hyh^{-1} \in A$.
\end{prop}

Using these rules, the following result was shown by H.~Zieschang in~\cite[Satz 1 and Korollar 1]{Zie1} and
refined by the third author in~\cite{Ros1,Ros2}.

\begin{thm}[The Nielsen Method for Amalgamated Free Products]$\mathstrut$\label{thm:3cases}\\
Let $G=H_1 \astA H_2$ be an amalgamated free product which satisfies Assumption~\ref{ass:A}.
Given a finite set of elements $\{ g_1,\dots,g_n\}$ in~$G$, there exists a 
finite sequence of shortening elementary Nielsen transformations which result 
in a set $\{x_1,\dots,x_n\}$ such that one of the following cases occurs.
\begin{enumerate}
\item[(1)] $x_i=1$ for some $i\in \{1,\dots,n\}$.

\item[(2)] Every $y \in \langle x_1,\dots,x_n\rangle$ has a
representation $y = x_{i_1}^{\epsilon_1} \cdots x_{i_q}^{\epsilon_q}$
with $\epsilon_j \in \{1,-1\}$, with $\epsilon_j = \epsilon_{j+1}$
in case $i_j = i_{j+1}$, and with $\lambda(x_{i_j}) \le \lambda(y)$ for $j=1,\dots,q$.

\item[(3)] For some $p\ge 1$, there is a subset $\{x_{i_1},\dots,x_{i_p} \}$
of $\{x_1,\dots,x_n\}$ which is contained in a subgroup~$H$ of~$G$ that is conjugate to~$H_1$
or~$H_2$, and such that $x_{i_1}\cdots x_{i_p}$ is conjugate to an element of $A\setminus \{1\}$.
\end{enumerate}
\end{thm}

Clearly, case~(1) implies that the rank of $\langle g_1,\dots,g_n \rangle$ is less than~$n$.
If $n\ge 2$ and case~(2) occurs then $\langle g_1,\dots,g_n\rangle$ is a non-trivial free product. 
Let us also interpret case~(3) in greater detail.

\begin{rem}\label{rem:case3}
Suppose that in the preceding theorem case~(3) occurs.
\begin{enumerate}
\item[(a)] If $G= \langle g_1,\dots,g_n\rangle$ then $p\ge 2$, since in this
case conjugations determine a Nielsen transformation.

\item[(b)] There exist $p\ge 2$ elements 
$\{x_{i_1}, \dots, x_{i_p}\}$ contained in~$H_1$ or~$H_2$ such that at least one of them
is not in~$A$ and $x_{i_1} \cdots x_{i_p} \in A \setminus \{1\}$.

\end{enumerate}
\end{rem}

Finally, we recall some refinements of the preceding material which were shown
in~\cite{KR} and will be useful later. For an element $g\in G$ with symmetric normal
form $g = \ell_1\cdots\ell_m\, k \, r_m \cdots r_1$ as in Proposition~\ref{prop:symmetric},
we write $\ell(g) = \ell_1\cdots \ell_m$ for the leading half, $k(g)=k$ for its kernel, and
$r(g)=r_m\cdots r_1$ for its rear half. Thus we have $g=\ell(g)\cdot k(g) \cdot r(g)$.

\begin{prop}\label{prop:refine}
Let $x,y\in G$ such that $r(x)=\ell(x)^{-1}$, such that $k(x)\in H_i \setminus A$
for some $i\in \{1,2\}$, and such that $\lambda(y) \le \lambda(x)$.
\begin{enumerate}
\item[(a)] For $\epsilon \in \{1,-1\}$ assume that $\lambda(xy^\epsilon) < \lambda(x)$ 
or $\lambda(y^{-\epsilon}x) < \lambda(x)$. Then one of the following cases occurs.
\begin{enumerate}
\item[(1)] $\lambda(y^{-\epsilon} x y^\epsilon) < \lambda(x)$;

\item[(2)] $r(y) = \ell(y)^{-1} = \ell(x)^{-1}$, and $k(x)k(y)^\epsilon$
or $k(y)^{-\epsilon} k(x)$ is contained in~$A$;

\item[(3)] $r(y) \ne \ell(y)^{-1}$, $k(y) \in H_i\setminus A$ for some $i\in \{1,2\}$,
$\lambda(x) = \lambda(y) = \lambda(y^{-\epsilon}x y^\epsilon)$, and~$A$ contains
$k(x)k(y)^\epsilon$ or $k(y)^{-\epsilon}k(x)$. In particular, we have $\min \{ \lambda(xy^\epsilon),
\lambda(y^{-\epsilon}x)\} < \lambda(y)$.

\end{enumerate}

\item[(b)] For $\epsilon \in \{1,-1\}$, assume that $\lambda(xy^\epsilon) = \lambda(x)$ 
or $\lambda(y^{-\epsilon}x) = \lambda(x)$. Then we have $\lambda(y^{-\epsilon}xy^\epsilon)\le 
\lambda(x)$.

\end{enumerate}
\end{prop}

\begin{proof}
Claim~(a) is shown in~\cite[Lemma 2.4]{KR} and claim~(b) in~\cite[Lemma 2.5]{KR}.
\end{proof}

%
%

\section{Low Rank Subgroups of Amalgamated Free Products}\label {sec3}

Let $n\ge 1$. Recall that a group~$G$ is called \textbf{$n$-generated} if it can be generated by~$n$ 
of its elements, and that the \textbf{rank} of~$G$, denoted by $\rk(G)$, 
is the least number~$n$ such that~$G$ is $n$-generated. The following notion
will play an important role in this paper.

\begin{defi}
For $n\ge 1$, a non-trivial group~$G$ is called an \textbf{$n$-free product of cyclics}
if every $n$-generated subgroup of~$G$ is a free product of at most~$n$ cyclic groups.
\end{defi}

Notice that this notion is completely different from the concept of an ``$n$-free group'' which
usually means that every $n$-generated subgroup is free (see~\cite{Hig}).

In the following we examine subgroups of an amalgamated free product
$G = H_1 \astA H_2$ which satisfies Assumption~\ref{ass:A}. In particular, let $H_1,H_2$
be finitely generated groups, and let $A = H_1\cap H_2$ be an infinite cyclic group which
is malnormal in~$G$.
In this setting, 2-generated subgroups of~$G$ can be characterized explicitly.
Note that the following result was first shown in~\cite[Thm.~6]{KS2} 
and is included here for completeness sake.

\begin{prop}[2-Generated Subgroups of Cyclically Amalgamated Free 
Products]$\mathstrut$\label{prop:2-gen}\\
Let $G=H_1 \astA H_2$ be an amalgamated free product as above, and assume that $\rk(H_1) \ge 2$
as well as $\rk(H_2) \ge 2$. If $H_1$ and~$H_2$ are 2-free products of cyclics, then~$G$
is a 2-free product of cyclics.

In particular, every 2-generated subgroup of~$G$ is a free product of $\le 2$ cyclics in this case.
\end{prop}

\begin{proof}
By Theorem~\ref{thm:3cases} and Remark~\ref{rem:case3}, every 2-generated subgroup of~$G$ 
is contained in a conjugate of~$H_1$ or~$H_2$, or it is a free product of cyclics. 
In the first case the claim follows
from the hypothesis that $H_1$ and~$H_2$ are 2-free products of cyclics.
\end{proof}

Our first main result characterizes 3-generated subgroups of~$G$.

\begin{thm}[3-Generated Subgroups of Cyclically Amalgamated Free 
Products]$\mathstrut$\label{thm:3-gen}\\
Let $G=H_1 \astA H_2$ be an amalgamated free product as above, and assume that
$\rk(H_1) \ge 3$ as well as $\rk(H_2)\ge 3$.
If $H_1$ and $H_2$ are 3-free products of cyclics, then $G$ is a 3-free product of cyclics.

In particular, every 3-generated subgroup of~$G$ is a free product of $\le 3$ cyclics in this case.
\end{thm}

\begin{proof}
Let $U = \langle g_1,g_2,g_3 \rangle$ be a 3-generated subgroup of~$G$, where $g_1,g_2,g_3\in G$.
If~$U$ is of rank~2, the claim follows from Proposition~\ref{prop:2-gen}. Hence we
assume that $\rk(U)=3$. We apply the sequence of Nielsen transformations given by Theorem~\ref{thm:3cases}
and obtain a new system of generators $\{x_1,x_2,x_3\}$ of~$U$.

Since $\rk(U)=3$, case~(1) of Theorem~\ref{thm:3cases} cannot occur. If case~(2) occurs, the
group~$U$ is a non-trivial free product. Thus we are left with analyzing case~(3), and we can apply
Remark~\ref{rem:case3}.b. Depending on whether we have $p=2$ or $p=3$ there, 
and interchanging the roles of~$H_1$
and~$H_2$ if necessary, we see that there are essentially two subcases:
\begin{enumerate}
\item[(3a)] $x_1,x_2,x_3 \in H_1$,
\item[(3b)] $x_1,x_2 \in H_1$ and $x_3 \notin H_1$.
\end{enumerate}
In subcase~(3a) the claim follows from the hypothesis that~$H_1$ and~$H_2$ are 3-free products
of cyclics. In subcase~(3b) it suffices to show that the subgroup $H=\langle H_1, x_3 \rangle$
of~$G$ is a 3-free product of cyclics, since then its subgroup $U=\langle x_1,x_2,x_3\rangle$
is a product of cyclics.

Using Proposition~\ref{prop:reduced}, we write the reduced normal form of~$x_3$ as
$x_3 = h_1 \cdots h_m \, a$, where the elements~$h_i$ are alternatingly from~$H_1$ and~$H_2$,
and where $a\in A$. Here we can include the element~$a$ in~$h_m$ if $h_m\in H_1$. If $h_m\in H_2$,
we increase~$m$ by one and use~$a$ as the new element~$h_m$. In either case, we can assume that
$x_3 = h_1\cdots h_m$ with $h_i$ alternatingly from ~$H_1$ and~$H_2$.

Next we consider $H = \langle H_1, x_3\rangle$. If $h_1 \in H_1$ then we use
$H = \langle H_1, h_1^{-1} x_3 \rangle$, and if $h_m \in H_1$ then we use
$H = \langle H_1, x_3 h_m^{-1}\rangle$. In this way we see that we can assume
$x_3 = h_1 \cdots h_m$ with $h_1,h_m \in H_2$. Since $A$ is malnormal in~$G$ and $A\subset H_1$,
it follows that $h_1,h_m \in H_2 \setminus H_1$.

Consequently, there are no shortening Nielsen transformations between elements of~$H_1$ and~$x_3$.
Thus an application of the Nielsen method shows that products in~$H$ have the property stated in
case~(2) of Theorem~\ref{thm:3cases}, \ie, that~$H$ is a 3-free product of cyclics.
\end{proof}

The second main result treats 4-generated subgroups of amalgamated products.

\begin{thm}[4-Generated Subgroups of Cyclically Amalgamated Free 
Products]$\mathstrut$\label{thm:4-gen}\\
Let $G=H_1 \astA H_2$ be an amalgamated free product as above, and assume that
$\rk(H_1) \ge 4$ as well as $\rk(H_2)\ge 4$.
If $H_1$ and $H_2$ are 4-free products of cyclics, then
every 4-generated subgroup of~$G$ is a free product of $\le 4$ cyclics 
or a one-relator quotient of a free product of four cyclics.
\end{thm}

\begin{proof}
Let $U = \langle g_1,g_2,g_3,g_4 \rangle$ be a 4-generated subgroup of~$G$, where $g_1,g_2,g_3,g_4\in G$.
If~$U$ is of rank~3, the claim follows from Theorem~\ref{thm:3-gen}. Hence we
assume that $\rk(U)=4$. We apply the sequence of Nielsen transformations given by Theorem~\ref{thm:3cases}
and obtain a new system of generators $\{x_1,x_2,x_3,x_4\}$ of~$U$.

Since $\rk(U)=4$, case~(1) of Theorem~\ref{thm:3cases} cannot occur. If case~(2) occurs, the
group~$U$ is a non-trivial free product. Thus we are left with analyzing case~(3), and we can apply
Remark~\ref{rem:case3}.b. Depending on whether $p=2$ or $p=3$ or $p=4$ there, 
and interchanging the roles of~$H_1$ and~$H_2$ if necessary, we see that there are essentially three subcases:
\begin{enumerate}
\item[(3a)] $x_1,x_2,x_3,x_4 \in H_1$;
\item[(3b)] $x_1,x_2,x_3 \in H_1$ and $x_4 \notin H_1$;
\item[(3c)] $x_1,x_2 \in H_1$ and $x_3,x_4 \notin H_1$.
\end{enumerate}
In subcase~(3a) the claim follows from the hypothesis that~$H_1$ and~$H_2$ are 4-free products
of cyclics. In subcase~(3b) we may argue analogously as in the proof of case~(3b) in 
Theorem~\ref{thm:3-gen}. More precisely, we consider $H = \langle H_1, x_4\rangle$
and write~$x_4$ in the form $x_4 = h_1 \cdots h_m$, where $h_1,h_m \in H_2 \setminus H_1$
and the elements $h_i$ are alternatingly from~$H_1$ and~$H_2$.
Again there are no shortening Nielsen transformations between elements of~$H_1$ and~$x_4$.
Consequently, an application of the Nielsen method shows that~$H$ is a 4-free product of cyclics.
Hence its subgroup $U = \langle x_1,x_2,x_3,x_4 \rangle$ is a free product of cyclics.

It remains to consider the case~(3c). The group $U = \langle x_1,x_2,x_3,x_4 \rangle$ is a
subgroup of $H = \langle H_1, x_3,x_4\rangle$. Arguing as in the proof of Theorem~\ref{thm:3-gen},
we may assume that we can write the reduced normal forms of~$x_3$ and~$x_4$ as
$x_3 = h_1 \cdots h_m$ and $x_4 = k_1 \cdots k_n$ where $m,n\ge 1$, where
$h_1,h_m,k_1,k_n \in H_2 \setminus A$, and where the elements $h_1,\dots,h_m$ as well as the elements
$k_1,\dots,k_n$ are alternatingly from~$H_1$ and~$H_2$.

By Theorem~\ref{thm:3-gen}, the subgroups $\langle x_1,x_2,x_3\rangle$
and $\langle x_1,x_2,x_4\rangle$ are free products of cyclics.
If $m\ge 2$, we form words in $H_1$, $x_3$, and $x_4$. 
By Theorem~\ref{thm:3cases} and Proposition~\ref{prop:refine}, essential new cancellations 
only can happen if~$x_3$ and~$x_4$ are involved. That is, there must be a subword which contains 
both~$x_3$ and~$x_4$ and which is conjugate to an element of~$A$, and~$x_3$ and~$x_4$ 
are in a proper conjugate the same factor~$H_1$ (resp.~$H_2$). 
Hence, since $m\ge 2$, this conjugate cannot produce essential cancellations between $H_1$, $x_3$ and~$x_4$.
Altogether, this implies that $U= \langle x_1,x_2,x_3,x_4\rangle$ is a free product of cyclics.

If $n\ge 2$, we argue analogously. Hence we are left with the case $m=n=1$, \ie, with
$x_3,x_4 \in H_2 \setminus A$. Recall that, by Assumption~\ref{ass:A}, the element
$u=\phi(e)$ is contained in~$H_1$. By Theorem~\ref{thm:3cases} and Proposition~\ref{prop:refine},
some word in $x_1,x_2$ has to be contained in $A\setminus \{1\}$. Since~$u$ generates~$A$,
this word is a power of~$u$. Hence there exists a
minimal exponent $k\ge 1$ such that $u^k \in \langle x_1,x_2\rangle$.

Next we recall that the element $v=\psi(e) \in H_2$ also generates~$A$ and satisfies
$v=u^{-1}$. Let us look at the subgroup $V = \langle v^{-k}, x_3,x_4\rangle$
of~$U = \langle x_1,x_2,x_3,x_4\rangle$. Since~$V$ is a 3-generated subgroup of~$G$,
we know from Theorem~\ref{thm:3-gen} that~$V$ is a free product of two or three cyclics.
Notice that~$V$ cannot be cyclic, as we have $\rk(U)=4$.
Moreover, recall that $\rk(H_2)\ge 4$ and that~$H_2$ is a 4-free product of cyclics by
the hypotheses.

First, we consider the case that $V = \langle v^{-k}, x_3, x_4 \rangle$ is a free product 
of three cyclics. Again we use the Theorem~\ref{thm:3cases} in combination with 
Proposition~\ref{prop:refine}. Since~$A$ is infinite and malnormal, we may assume that 
no element $v^{-i}$ with $1 \le i \le k-1$ is in $\langle x_3,x_4\rangle$. 
Here we may argue in a symmetric manner by suitable conjugations 
and eventual replacings. 

Now we consider non-trivial words in $x_1,x_2,x_3,x_4$ by viewing them blockwise in a combination 
of cases~(2) and~(3) of Theorem~\ref{thm:3cases}. These words are never equal to the identity element, 
in other words, the elements of $\langle x_1,x_2\rangle$ do not influence the 
generators of~$\langle x_3,x_4\rangle$. Therefore $U=\langle x_1,x_2,x_3,x_4\rangle$
is a free product of cyclics in this case.

Consequently, we are now left with the case when $V = \langle a\rangle \ast \langle b\rangle$
is a free product of two cyclics, where $a,b\in V$. 
We distinguish the following two situations.

\medskip
{\it Situation I:}\/ There exist a number $\ell \in \{1,\dots,k-1\}$ and an element
$c \in H_2 \setminus \{1\}$ such that $c v^{-\ell} c^{-1} \in V = \langle a\rangle
\ast \langle b\rangle$.  

In this situation we replace~$a$ by $a' = c^{-1}ac$, we replace~$b$ by $b'=c^{-1}bc$,
we replace~$x_1$ by $x_1' = c^{-1}xc$, and we replace~$x_2$ by $x_2' = c^{-1} x_2 c$.
The group $U' = \langle x_1',x_2',a',b'\rangle$ is conjugate to 
\[
\langle x_1,x_2,a,b\rangle \;=\; \langle x_1,x_2, V \rangle \;=\; \langle x_1,x_2,x_3,x_4 \rangle
\;=\; U
\]
and has therefore the same combinatorial structure. Thus it suffices to prove that~$U'$ has the
claimed structure. After we apply the Nielsen method to~$U'$, we may assume that
$a',b' \in H_2$ and $x_1',x_2' \in H_1$. If we are still in Situation~I, but with a smaller
number~$\ell$, we repeat the construction. Finally, we will reach Situation~II.

\medskip
{\it Situation II:}\/ There is no number $\ell\in \{1,\dots,k-1\}$ with
$cv^{-\ell}c^{-1} \in V$ for some $c\in H_2 \setminus \{1\}$.

In this situation we use the fact that~$V = \langle v^{-k}, x_3,x_4\rangle$ is a free product
of two cyclics to conclude that there is a defining relation $R(v^{-k},x_3,x_4)=1$. Since $v^{-k}=u^k$
is also a word in $x_1,x_2$, we can write this as a relation $R'(x_1,x_2,x_3,x_4)=1$.
Hence $U=\langle x_1,x_2,x_3,x_4\rangle$ is a one-relator quotient
of a free product of four cyclics, and the proof is complete.
\end{proof}

A special case of the preceding two theorems arises when~$G$ is a group of F-type.
Let us first recall the definition.

\begin{defi}
A group~$G$ is said to be \textbf{of F-type} if it admits a presentation of the form
\[
G = \langle a_1,\dots,a_n \mid a_1^{e_1} = \cdots = a_n^{e_n} = U(a_1,\dots,a_p)\, 
V(a_{p+1},\dots,a_n) = 1\rangle,
\]
where $n\ge 2$, $e_1,\dots,e_n \in \mathbb{N} \setminus \{1\}$, $p\in \{1,\dots,n-1\}$, and
where $U(a_1,\dots,a_p)$ and $V(a_{p+1},\dots,a_n)$ are cyclically reduced words of infinite order
in the free product $\langle a_1,\dots,a_n \mid a_1^{e_1} = \cdots = a_n^{e_n} =1 \rangle$.
\end{defi}

At this point we now suppose that~$G$ does not reduce to a free product of cyclics and that $U(a_1, \dots, a_p)\, 
V(a_{p+1}, \dots, a_n)$ involves all the generators.
Then~$G$ decomposes as a non-trivial amalgamated free product $G=H_1 \astA H_2$
as above, where
\begin{align*}
H_1 &\;=\; \langle a_1,\dots,a_p \mid a_1^{e_1} = \cdots = a_p^{e_p} = 1 \rangle, &\\
H_2 &\;=\; \langle a_{p+1},\dots,a_n \mid a_{p+1}^{e_{p+1}} = \cdots = a_n^{e_n} = 1 \rangle, &\\
A &\;=\; \langle U(a_1,\dots,a_p) \rangle = \langle V(a_{p+1},\dots,a_n)^{-1} \rangle.
\end{align*}

The following corollary follows from Proposition~\ref{prop:2-gen}, Theorem~\ref{thm:3-gen},
and Theorem~\ref{thm:4-gen}.

\begin{cor}[Low Rank Subgroups of Groups of F-Type]$\mathstrut$\label{cor:F-type}\\
Let $G$ be a group of F-type as above. Suppose that~$A$ is malnormal in~$G$. 
Then the following claims hold:
\begin{enumerate}
\item[(a)] If $p\ge 2$ and $n-p\ge 2$ then every 2-generated subgroup of~$G$ 
is a free product of cyclics.

\item[(b)] If $p\ge 3$ and $n-p\ge 3$ then every 3-generated subgroup of~$G$
is a free product of cyclics.

\item[(c)] If $p\ge 4$ and $n-p \ge 4$ then every 4-generated subgroup of~$G$
is a free product of cyclics or a one-relator quotient of a free product of four cyclics.
\end{enumerate}
\end{cor}

Parts~(a) and~(b) of this corollary were previously shown in~\cite{FRW}.
The next example shows that part~(c) is optimal.

\begin{exa}
Consider the group
\[
G \;=\; \langle a_1,\dots,a_6 \mid a_1^2 = a_2^2 = a_3^2 = a_4^2 = a_5^2 = a_6^3 = 
a_1 \cdots a_6 =1 \rangle .
\]
It is of F-type with $n=6$, $p=3$, with $U(a_1,a_2,a_3) = a_1 a_2 a_3$, and with 
$V(a_4,a_5,a_6) = a_6^{-1} a_5^{-1} a_4^{-1}$. Moreover, notice that the 
subgroup $A= \langle U(a_1,a_2,a_3) \rangle
= \langle V(a_4,a_5,a_6)^{-1} \rangle$ is malnormal in~$G$. 
Therefore Corollary~\ref{cor:F-type}.b
says that every 3-generated subgroup of~$G$ is a free product of cyclics.

Now we consider the 4-generated subgroup $H=\langle x_1,x_2,x_3,x_4\rangle$ of~$G$, where
$x_1 = a_1 a_2$, $x_2 = a_1 a_3$, $x_3 = a_1 a_4$, and $x_4 = a_1 a_5$. 
Then we calculate
\[
x_1 x_2^{-1} x_3 x_4^{-1} x_1^{-1} x_2 x_3^{-1}x_4 \;=\; x_6^{-2} \;=\; x_6 \in H, 
\]
and thus $a_5 = a_5^{-1} = a_6 x_1 x_2^{-1} x_3 \in H$, $a_1=x_4 a_5 \in H$, $a_2 = a_1x_1 \in H$,
$a_3 = a_1 x_2 \in H$, and $a_4=a_1 x_3 \in H$. This proves $H=G$.
However, the group~$G$ is neither a free product of cyclics nor a one-relator 
quotient of a free product of four cyclics. 

It can be shown that the claim of Corollary~\ref{cor:F-type}.c holds if
\begin{enumerate}
\item[(1)] $p\ge 3$, $n-p \ge 3$, and at most $n-2$ of the exponents~$e_i$
are equal to~2, or if

\item[(2)] $p=3$, $n-p \ge 4$, or if $p\ge 4$, $n-p=3$.
\end{enumerate}
In this sense, Corollary~\ref{cor:F-type}.c is optimal.
\end{exa}

The final remark in this section connects it to the next one.

\begin{rem}
Let~$G$ be a group of F-type which is not a free product of cyclics.
Assume that every subgroup of~$G$ of infinite index is a free product of cyclics.
Then~$G$ is a co-compact planar discontinuous group (\cf~\cite[Thm.~3.5.45]{FMRSW}).
\end{rem}

%
%

\section{Free Subgroups of Infinite Index}\label{sec4}

In this section we turn our attention to free subgroups of infinite index
in cyclically amalgamated free products and in certain HNN extensions. 
The following material is largely known and should be considered as a survey
of the application of purely group theoretic methods to the topic at hand.

The central open question for free subgroups of 
infinite index is the {\it Surface Group Conjecture}.
The first version of this conjecture was formulated 1980 in the Kourovka Notebook
by O.V. Melnikov (\cf~\cite[Problem 7.36]{KM}). It has since been disproved, but numerous improved 
versions were formulated over time (for instance, see~\cite{FKMRR,CFR,FRW}), and the 
latest one reads as follows (\cf~\cite{Wil2}).

\begin{conj}[{\bf The Surface Group Conjecture}]$\mathstrut$\label{conj:SG}\\
Let $G$ be an infinite, finitely presented group, such that every
subgroup of infinite index is free. Must $G$ be isomorphic to either a free group
or a surface group?
\end{conj}

Here a \textbf{surface group} is a group isomorphic to the fundamental group
of an orientable surface of genus~$g$, \ie,
\[
G \;\cong\; \langle a_1,b_1, a_2,b_2,\dots,a_g,b_g \mid [a_1,b_1] = \cdots = [a_g,b_g] = 1 \rangle
\]
or of a non-orientable surface of genus~$g$, \ie,
\[
G \;\cong\; \langle a_1,\dots,a_g \mid a_1^2 \cdots a_g^2 = 1 \rangle .
\]

A celebrated result of H.\ Wilton is that the surface group conjecture
holds for infinite one-relator groups (\cf~\cite{Wil2}). Its proof is based on a number of deep 
results in topology. Another important situation where this conjecture holds is
the case of fully residually free groups (see~\cite{CFR}).

For the following kind of amalgamated product, a chiefly group theoretic 
proof of the conjecture is available.

\begin{thm}[The Surface Group Conjecture for Cyclically Amalgamated Free 
Products]$\mathstrut$\label{thm:SGCforAFP}\\
Let $G$ be a finitely generated group which is an amalgamated free product
$G = H_2 \astA H_2$, where~$A$ is an infinite cyclic group of infinite index 
in~$H_1$ and~$H_2$. Assume that $G$ is not free, but every subgroup of infinite 
index in~$G$ is free. Then the Surface Group Conjecture~\ref{conj:SG} holds for~$G$.
\end{thm}

\begin{proof}
Since~$A$ is infinite and has infinite index in~$H_1$ and~$H_2$, the factors~$H_1$
and~$H_2$ have infinite index in~$G$. By assumption, it follows that~$H_1$ and~$H_2$ are
free groups. Thus~$G$ is a free product of two free groups amalgamated in an infinite 
cyclic subgroup and is also known as a \textbf{cyclically pinched one-relator
group}. 

For this kind of group, \ie, for $G= \langle H_1,H_2;\, u=v\rangle$, where~$H_1$ and~$H_2$ are free
groups and where $u\in H_1$ and $v\in H_2$,
the surface group conjecture was shown in~\cite[Thm.~3.1.1]{CFR}.
For the convenience of the reader, we briefly recall the main steps.
We distinguish four cases depending on whether~$u$ and~$v$ are proper powers of elements of~$G$.

\medskip
\noindent{\it Case 1:}\/ Assume that there exist $g,h\in G$ such that $u=g^2$ and $v=h^2$.
If already $G = \langle g,h \mid g^2 = h^2\rangle$, then~$G$ is a non-hyperbolic, non-orientable
surface group of genus~2. If this is not the case then
$\langle g,h \mid g^2 = h^2\rangle$ is a proper subgroup of~$G$, and hence~$G$ 
contains a non-free subgroup of infinite index.

\medskip
\noindent{\it Case 2:}\/  Assume that there exist $g,h\in G$ such that $u=g^m$ and $v=h^n$
with $m\ge 2$, $n\ge 2$, and $m+n\ge 5$. Then the subgroup $U = \langle g^m, gh\rangle$
is a free abelian subgroup of rank~2 of~$G$, in contradiction to Proposition~\ref{prop:2-gen}.

\medskip
\noindent{\it Case 3:}\/ Assume that $u$ is no proper power in~$G$ and $v=h^n$ for some $h\in G$
and $n\ge 2$. Then $u = v^{-1} = h^{-n}$ yields a contradiction.

\medskip
\noindent{\it Case 4:}\/ Assume that neither~$u$ nor~$v$ is a proper power in~$G$.
In this case,~\cite[Thm.~A]{JR} shows that~$G$ is a word hyperbolic group.
Next we use~\cite[Thm.~4.11]{Sta} to show that~$G$ is one-ended. In fact, if~$G$
has no ends, then~$G$ is finite. If~$G$ has two ends, then~$G$ has an infinite
cyclic subgroup of finite index. Finally, if~$G$ has more than two ends, it has infinitely
many ends and is a non-trivial free product. As all of these cases are impossible for the
groups~$G$ under consideration, they have to be one-ended.
The final step is an application of H.\ Wilton's result that a hyperbolic, one-ended
cyclically pinched one-relator group, all of whose subgroups of infinite index are free,
is a surface group (\cf~\cite[Corollary~4]{Wil1}).
\end{proof}

It is also natural to ask for an analogue of this theorem for HNN extensions.
Recall that a group~$G$ is called an \textbf{HNN extension} of a group~$H$
if there exist subgroups $A,B$ of~$H$ and an isomorphism $\phi:\; A \longrightarrow B$ 
such that
\[
G \;=\; \langle H, t \mid t\, a\, t^{-1} = \phi(a) \text{\ \rm for all\ }a\in A \rangle .
\]
HNN extensions are natural analogues of amalgamated free products. For instance, they
arise as the fundamental groups of unions of topological spaces where the intersection 
is not connected.

Not much seems to be known for low rank subgroups and free subgroups 
of infinite index of HNN extensions. For 2-generated and 3-generated subgroups, 
the following result from~\cite{FRR} is the best we know. If we drop
the assumption that~$u$ is not conjugate in~$H$ to~$v$ or~$v^{-1}$,
a slightly stronger version for 2-generated subgroups of~$G$ is given 
in~\cite[Thm.~1.6.42]{FMRSW}.

\begin{thm}[Low Rank Subgroups of HNN Extensions]$\mathstrut$\label{thm:23HNN}\\
Let $H$ be a free group, let $u,v\in H$ be non-trivial elements, none of which is a proper
power in~$H$, and assume that $u$ is not conjugate in~$H$ to~$v$ or~$v^{-1}$. 
Let $G = \langle H, t \mid tut^{-1} =v \rangle$ be the corresponding HNN extension of~$H$.
\begin{enumerate}
\item[(a)] Every 2-generated subgroup of~$G$ is either abelian
or a free group of rank two.

\item[(b)] Every 3-generated subgroup of~$G$
is either free of rank $\le 3$ or has a one-relator presentation for its three generators.
\end{enumerate}
\end{thm}

Notice that HNN extensions of the special type in this theorem are also called
\textbf{conjugacy pinched one-relator groups}. Recall that a 2-generated subgroup $\langle x,y\rangle$
of~$G$ is called \textbf{maximal} if it is not properly contained in another 2-generated subgroup,
and it is called \textbf{strongly maximal} if for every $g\in G$ there exists an element
$h\in G$ such that $\langle x, gyg^{-1} \rangle \subseteq \langle x,hyh^{-1}\rangle$ and
$\langle x, hyh^{-1}\rangle$ is a maximal 2-generated subgroup of~$G$.

\begin{cor}
In the setting of the theorem, assume that the subgroup $\langle u,v\rangle$ of~$G$
is strongly maximal. Then~$G$ is 3-free, \ie, every 3-generated subgroup of~$G$ is free.
\end{cor}

As regards to the surface group conjecture, we can show it for the following
type of HNN extensions. Notice that the following result is slightly stronger
than~\cite[Thm.~3.1.2]{CFR} and~\cite[Thm.~1.7.22]{FMRSW}.

\begin{thm}[The Surface Group Conjecture for HNN Extensions]$\mathstrut$\label{thm:SGCforHNN}\\
Let $H$ be a finitely generated group, let $u,v\in H$ be non-trivial,
and let~$G$ be the HNN extension $G= \langle t,H \mid tut^{-1} = v \rangle$.
Assume that $G$ is not free, but every subgroup of infinite index in~$G$ is free.
Then the Surface Group Conjecture~\ref{conj:SG} holds for~$G$.
\end{thm}

\begin{proof}
Let~$N$ be the normal closure of~$H$ in~$G$. Then $G/N$ is an infinite cyclic group.
Hence the group~$H$ is of infinite index in~$G$, and the hypothesis implies that~$H$
is a free group. Thus~$G$ is a conjugacy pinched one-relator group.
For such groups, the surface group conjecture can be shown based in part on~\cite[Thm.~3.1.2]{CFR}.
Let us follow the main steps and complete the proof at the appropriate place.
We distinguish two cases depending on whether~$u$ or~$v$ are proper powers of elements of~$G$.

\smallskip
\noindent{\it Case 1:}\/ Assume that $u = g^m$ and $v = h^n$ with $g,h\in G$ and with $m,n\ge 2$.
If $G = \langle g,h \mid g^2 = h^2\rangle$ then~$G$ is a surface group 
and contains a free abelian subgroup of rank~2. Now assume that~$G$ is not isomorphic 
to  $\langle g,h \mid g^2 = h^2 \rangle$. 
Using the normal form in HNN groups (\cf~\cite[Thm.~1.6.4]{FMRSW}), it follows that
the subgroup generated by $w=tgt^{-1}$ and~$h$ has the presentation $W=\langle w,h \mid w^m = h^n\rangle$.
Then the subgroup $\langle w^m,wh\rangle$ of~$W$ is a free abelian group of rank~2 and has infinite
index in~$G$, in contradiction to the hypothesis.

\smallskip
\noindent{\it Case 2:}\/ Assume that~$u$ or~$v$ is not a proper power in~$G$.
Here we distinguish two subcases.

\smallskip
\noindent{\it Subcase 2a:}\/ Assume that there exists an element $x\in H$ such that
the intersection 
$\langle u\rangle \cap x\langle v\rangle x^{-1}$ is infinite, and thus infinite cyclic.
After a suitable conjugation and possibly interchanging~$u$ and~$v$, we may assume 
that~$u$ is not a proper power in~$H$ and $v=u^k$ with $k\ne 0$.
Using the normal form, we can see that the subgroup of~$G$ generated by~$u$ and~$t$
has the presentation $U = \langle u,t \mid tut^{-1} = u^k\rangle$. 

In the case $\rk(H)\ge 2$, let~$V$ be the normal closure of~$U$ in~$G$, and let 
$H=\langle a_1,\dots,a_n\rangle$. Then $G/V  \cong \langle a_1,\dots,a_n \mid u=1\rangle$
is infinite by the Freiheitssatz. Hence~$U$ is a non-free subgroup of infinite index in~$G$,
in contradiction to the hypothesis.

In the case $\rk(H)=1$, we let $H=\langle a\rangle$. Then we get $G=\langle a, t \mid tat^{-1} 
= a^k\rangle = B(1,k)$ with a Baumslag-Solitar group $B(1,k)$. For $\vert k \vert = 1$, this
is a surface group, and for $p=\vert k \vert \ge 2$, the group $B(1,p)$ is isomorphic 
to $\mathbb{Z}[\frac{1}{p}]
\rtimes \mathbb{Z}$ (\cf~\cite{FM}) and contains $\mathbb{Z}[\frac{1}{p}]$ 
as a non-free subgroup of infinite index. Hence, also in the case $G\cong B(1,k)$ 
with $\vert k \vert \ge 2$, the group~$G$ has a non-free subgroup
of infinite index, in contradiction to the hypothesis.

\smallskip
\noindent{\it Subcase 2b:}\/ Assume that $\langle u\rangle \cap x \langle v\rangle x^{-1}$
is finite for every $x\in H$. Since~$G$ is torsion free, we get 
$\langle u\rangle \cap x \langle v\rangle x^{-1} = \{1\}$. By~\cite{KMy}, it follows 
that~$G$ is word hyperbolic. As in the proof of Theorem~\ref{thm:SGCforAFP}, we see that~$G$
is one-ended, and then an application of~\cite{Wil1} finishes the proof.
\end{proof}

\bibliographystyle{plain}
\bibliography{SubgrOfAmalgFP.bib}

\end{document}